\documentclass[a4paper]{amsproc}
\usepackage{amssymb}
\usepackage{amscd}

\theoremstyle{plain}
 \newtheorem{thm}{Theorem}[section]
 \newtheorem{prop}{Proposition}[section]
 \newtheorem{lem}{Lemma}[section]
 \newtheorem{cor}{Corollary}[section]
\theoremstyle{definition}

\numberwithin{equation}{section}

\def\ji {\char'032}
\def\ja {\char'037}
\def\m  {\char'176}

\font\srit=wncyi8
 \font\srrm=wncyr8

\newcommand{\R}{\mathbb{R}}
\newcommand{\C}{\mathbb{C}}
\newcommand{\zkj}{\bar{z}_j}
\newcommand{\dzj}{\frac{\partial}{\partial z_j}}
\newcommand{\dzkj}{\frac{\partial}{\partial\bar{z}_j}}
\newcommand{\ds}{\displaystyle}

\newcommand{\T}{\mathbb{T}}

\DeclareMathOperator{\tr}{\mathrm{tr}}
\newcommand{\ad}{\mathrm{ad}}

\DeclareMathOperator{\diag}{\mathrm{diag}}

\DeclareMathOperator{\pr}{\mathrm{pr}}

\DeclareMathOperator{\orth}{\mathrm{orth}}

\begin{document}
\title[CONTACT FLOWS AND INTEGRABLE SYSTEMS]
{CONTACT FLOWS AND INTEGRABLE SYSTEMS}

\author{ Bo\v zidar Jovanovi\'c \and
Vladimir Jovanovi\'c}


\maketitle

\centerline{\small Mathematical Institute SANU, Serbian Academy of
Sciences and Arts} \centerline{\small Kneza Mihaila 36, 11000
Belgrade, Serbia} \centerline{\small E-mail: bozaj@mi.sanu.ac.rs }

\

\centerline{\small Faculty of Sciences, University of Banja Luka}
\centerline{\small Mladena Stojanovi\'ca 2, 51000 Banja Luka,
Bosnia and Herzegovina}\centerline{\small E-mail: vlajov@blic.net}

\

\begin{abstract}
We consider Hamiltonian systems restricted to the hypersurfaces of
 contact type and obtain a partial version of the
Arnold-Liouville theorem: the system not need to be integrable on
the whole phase space, while the invariant hypersurface is
foliated on an invariant Lagrangian tori. In the second part of
the paper we consider contact systems with constraints. As an
example, the Reeb flows on Brieskorn manifolds are considered.
\end{abstract}

\tableofcontents

\section{Introduction}

Usually, integrable systems are considered within a framework of
symplectic or Poisson geometry, but there is a well defined
non-Hamiltonian (e.g., see \cite{Bo, Zu}) as well as a contact
setting studied in \cite{BM, Web, Le, KT, Jo, Boyer}. The aim of
this paper is to stress some natural applications of contact
integrability to the Hamiltonian systems and to provide examples
of contact integrable flows.

In the first part of the paper we consider Hamiltonian systems
restricted to the hypersurfaces of contact type and obtain a
partial version of the Arnold-Liouville theorem: the system need
not be integrable on the whole phase space, while the invariant
hypersurface is foliated on invariant Lagrangian tori with
quasi-periodic dynamics. The construction can also be applied to
the partially integrable systems and the systems of
Hess–-Appel'rot type, where the invariant hypersurface is foliated
on Lagrangian tori, but not with quasi-periodic dynamics \cite{DG,
Jo1, DGJ}.

In the second part of the paper we consider contact systems with
constraints and derive a contact version of Dirac's construction
for constrained Hamiltonian systems. As an example, the Reeb flows
on Brieskorn manifolds are considered. From the point of view of
integrability, those Reeb flows are very simple. However, we use
them to  clearly demonstrate the use of Dirac's construction and
to interpret the regularity of the Reeb flows within a framework
of contact integrability.

For the completeness and clarity of the exposition, in Subsections
1.1, 1.2 and 2.1 we recall basic definitions in contact geometry,
the notion of contact noncommutative integrability, and the
Maupertuis--Jacobi metric, respectively.

\subsection{Contact flows and the Jacobi bracket}   A {\it contact form}
$\alpha$ on a $(2n+1)$-dimensional ma\-nifold $M$ is a Pfaffian form
satisfying $\alpha\wedge(d\alpha)^n\ne0$. By a {\it contact
manifold} $(M,\mathcal H)$ we mean a connected
$(2n+1)$-dimensional manifold $M$ equipped with a nonintegrable
{\it contact} (or {\it horizontal}) {\it distribution} $\mathcal
H$, locally defined by a contact form: $\mathcal
H\vert_U=\ker\alpha\vert_U$, $U$ is an open set in $M$ \cite{LM}.
Two contact forms $\alpha$ and $\alpha'$ define the same contact
distribution $\mathcal H$ on $U$ if and only if $\alpha'=a\alpha$
for some nowhere vanishing function $a$ on $U$.

The condition $\alpha\wedge(d\alpha)^n\ne0$ implies that the form
$d\alpha\vert_x$ is nondegenerate (symplectic) structure
restricted to $\mathcal H_x$. The conformal class of
$d\alpha\vert_x$ is invariant under the change $\alpha'=a\alpha$.
If $\mathcal V$ is a linear subspace of $\mathcal H_x$, then we
have well defined its orthogonal complement in $\mathcal H_x$ with
respect to $d\alpha\vert_x$, as well as the notion of the {\it
isotropic} ($\mathcal V \subset \orth_\mathcal H \mathcal V$),
{\it coisotropic} ($\mathcal V \supset \orth_\mathcal H \mathcal
V$) and the {\it Lagrange} subspaces ($\mathcal V =\orth_\mathcal
H \mathcal V$) of $\mathcal H_x$ . A submanifold $N\subset M$ is
{\it pre-isotropic} if it transversal to $\mathcal H$ and if
$\mathcal G_x=T_x N \cap\mathcal H_x$ is an isotropic subspace of
$\mathcal H_x$, $x\in N$. The last condition is equivalent to the
condition that distribution $\mathcal G=\bigcup_x G_x$ defines a
foliation. It is {\it pre-Legendrian} submanifold if it is of
maximal dimension $n+1$, that is $\mathcal G$ is a Lagrangian
subbundle of $\mathcal H$.

A {\it contact diffeomorphism} between contact manifolds
$(M,\mathcal H)$ and $(M',\mathcal H')$ is a diffeomorphism $\phi:
M\to M'$ such that $\phi_*\mathcal H=\mathcal H'$. An {\it
infinitesimal automorphism} of a contact structure $(M,\mathcal
H)$ is a vector field $X$, called a {\it contact vector field}
such that its local 1-parameter group is made of contact
diffeomorphisms. Locally, if $\mathcal H=\ker\alpha$, then
$\mathcal L_X\alpha=\lambda\alpha$, for some smooth function
$\lambda$.

From now on, we consider {\it co-oriented (or strictly) contact
manifolds} $(M,\alpha)$, where contact distributions $\mathcal H$
are associated  to globally defined contact forms $\alpha$. The
{\it Reeb vector field} $Z$ is a vector field uniquely defined by
\begin{equation}\label{REEB}
i_Z\alpha=1, \qquad i_Z d\alpha=0.
\end{equation}

The tangent bundle $TM$ and the cotangent bundle $T^*M$ are
decomposed into $TM=\mathcal Z \oplus \mathcal H$ and
$T^*M=\mathcal H^0\oplus \mathcal Z^0$,
 where $\mathcal Z=\R Z$ is the kernel of
$d\alpha$, $\mathcal Z^0$ and $\mathcal H^0=\R\alpha$ are the
annihilators of $\mathcal Z$ and $\mathcal H$, respectively.  The
sections of $\mathcal Z^0$ are called {\it semi-basic forms}.
Whence, we have decompositions of vector fields and 1-forms
\begin{equation}
X=(i_X\alpha)Z+\hat X, \qquad \eta=(i_Z\eta)\alpha+\hat \eta,
\label{decomp*}
\end{equation}
where $\hat X$ is horizontal and $\hat\eta$ is semi-basic.

The mapping $ \alpha^\flat: X \mapsto -i_X d\alpha$ carries $X$
into a semi-basic form. The form $d\alpha$ is non-degenerate on
$\mathcal H$ and the restriction of $\alpha^\flat$ to horizontal
vector fields is an isomorphism whose inverse will be denoted by
$\alpha^\sharp$. The vector field
\begin{equation}
Y_f = f Z+\hat{Y}_f, \qquad \hat{Y}_f=\alpha^\sharp(\widehat{df}).
\label{iso}
\end{equation}
is a contact vector field ($\mathcal L_{Y_f}\alpha=Z(f)\alpha$)
and
\begin{equation}\label{eq_f}
\dot x=Y_f
\end{equation}
is called the {\it contact Hamiltonian equation} corresponding to
$f$. Notice that $Z=Y_1$ and $f=i_{Y_f}\alpha$.

The mapping $f\mapsto Y_f$ is a Lie algebra isomorphism
($Y_{[f,g]}=[Y_f,Y_g]$) between the set $C^\infty(M)$ of smooth
functions on $M$ and the set $\mathcal N$ of infinitesimal contact
automorphisms. Here, the {\it Jacobi bracket} $[\cdot,\cdot]$ on
$C^\infty(M)$ is defined by (see \cite{LM})
\begin{equation}\label{der}
[f,g]=d\alpha(Y_f,Y_g)+fZ(g)-gZ(f)={Y_f}(g)- g Z(f).
\end{equation}

The Jacobi bracket does not satisfy the Leibniz rule. However,
within the class of functions that are integrals of the Reeb flow, $
C^\infty_\alpha(M)=\{f\in C^\infty(M)\,\vert\, Z(f)=[1,f]=0\}, $
the Jacobi bracket has the usual properties of the Poisson bracket:
$g$ is an integral of \eqref{eq_f} if and only if $[g,f]=0$ and if
$g_1$ and $g_2$ are two integrals of \eqref{eq_f}, then
$[g_1,g_2]$ is also an integral.

\subsection{Contact integrability}
A Hamiltonian system on $2n$-dimensional symplectic manifold
$(P,\omega)$ is {\it noncommutatively integrable} if it has $2n-r$
almost everywhere independent integrals $F_1,F_2,\dots,F_{2n-r}$
and $F_1,\dots,F_r$ commute with all integrals
\begin{equation}\label{nehorosev}
\{F_i,F_j\}=0, \quad j=1,\dots,r, \quad i=1,\dots,2n-r.
\end{equation}
Regular compact connected invariant manifolds of the system are
isotropic tori, generated by Hamiltonian flows of $F_1,\dots,F_r$.
In a neighborhood of a regular torus, there exist canonical {\it
generalized action--angle coordinates} such that integrals $F_i$,
$i=1,\dots,r$ depend only on actions and the flow is translation
in angle coordinates (see Nekhoroshev \cite{N} and Mishchenko and
Fomenko \cite{MF}). If $r=n$ we have the usual commutative
integrability described in the Arnold-Liouville theorem \cite{Ar}.

Noncommutative integrability of contact flows  can be stated in
the following form (see \cite{Jo}). Suppose we have a collection
of integrals $f_1,f_2,\dots,f_{2n-r}$ of equation \eqref{eq_f} on
a $(2n+1)$-dimensional co-oriented contact manifold $(M,\alpha)$,
where $f=f_1$ or $f=1$ and
\begin{equation}\label{involucija}
[1,f_i]=0, \quad [f_i,f_j]=0, \quad i=1,\dots,2n-r, \quad
j=1,\dots, r.
\end{equation}

Let $T$ be a compact connected component of the level set $
\{f_1=c_1,\dots,f_{2n-r}=c_{2n-r}\} $ and $df_1 \wedge \dots
\wedge df_{2n-r}\ne 0 $ on $T$. Then $T$ is diffeomorphic to a
pre-isotropic $r+1$-dimensional torus $\mathbb T^{r+1}$. There
exist a neighborhood $U$ of $T$ and a diffeomorphism $\phi: U\to
\mathbb T^{r+1} \times D$,
\begin{equation}\label{action-angle}
\phi(x)=(\theta,y,x)=(\theta_0,\theta_1,\dots,\theta_r,y_1,\dots,y_r,x_1,\dots,x_{2s}),
\quad s=n-r,
\end{equation}
where $D\subset \R^{2n-r}$, such that

(i) $\alpha$ has the following canonical form
\begin{equation}\label{canonical}
\alpha_0=(\phi^{-1})^*\alpha=
y_0d\theta_0+y_1d\theta_1+\dots+y_rd\theta_r+g_1dx_1+\dots+g_{2s}
dx_{2s},
\end{equation}
where $y_0$ is a smooth function of $y$ and $g_1,\dots,g_{2s}$
are functions of $(y,x)$;

(ii) the flow of $Y_f$ on invariant tori is quasi-periodic
\begin{equation}\label{namotavanje}
(\theta_0,\theta_1,\dots,\theta_r) \longmapsto
(\theta_0+t\omega_0,\theta_1+t\omega_2,\dots,\theta_r+t\omega_r),
\quad t\in\R,
\end{equation}
where the frequencies $\omega_0,\dots,\omega_r$ depend only on
$y$.

Note that we can define the notion of noncommutative integrability
of contact equations $\dot x=X$ on contact manifolds $(M,\mathcal
H)$ without involving globally defined contact form $\alpha$ (for more
details see \cite{Jo}).

The case $r=n$ corresponds to the contact commutative integrability
defined by Banyaga and Molino, where the invariant tori are
pre-Legendrian \cite{BM} (see also \cite{KT, Boyer}). A slightly
different notion of (commutative) contact integrability in terms
of the {\it modified Poisson} (or the {\it Jacobi-Mayer}) {\it
bracket} $(\cdot,\cdot)$ \cite{Gr},
$$
(f,g)\, \alpha\wedge (d\alpha)^n=n\, df\wedge dg
\wedge\alpha\wedge(d\alpha)^{n-1},
$$
is given by Webster \cite{Web}. Note that the modified Poisson
bracket satisfies the Leibniz rule, but does not satisfy the
Jacobi identity.

\subsection{Outline and results of the paper}
In Section 2 we firstly consider integrable Hamiltonian systems
restricted to the hypersurfaces of contact type and prove that the
appropriate Reeb flows are examples of contact integrable systems
(Proposition \ref{prva}). This is a contact analogue of the
construction of integrable geodesic flows by the use of  the
Maupertuis principle \cite{BKF}. As an example, a contact
algebraic hypersurface of degree 4 in $(\R^{2n+2}(q,p), dp\wedge
dq)$ with an integrable Reeb flow associated to the
Gelfand--Cetlin system on $u(n+1)$ is given (Proposition 2.1).

Proposition \ref{prva} is a special case of a more general
statement on isoenergetic integrability (Theorem \ref{druga},
Corollary \ref{IAL}), where the system doesn't need to be integrable
on the whole phase space, but the relation \eqref{nehorosev} holds
on a fixed isoenergetic hypersurface $M$ of contact type. Then, as
in the case of the usual (non-commutative) integrability, $M$ is
foliated on invariant (isotropic) Lagrangian tori with
quasi-periodic dynamics.

The construction of partial integrals for natural mechanical
systems on a fixed isoenergetic hypersurface is studied by
Birkhoff \cite{Bir0, Bir}. Although the problem is classical and
some variants of restricted integrability on invariant manifolds
are already studied (see  \cite{N3} and references therein),
Theorem \ref{druga} is a benefit of a noncommutative contact
integrability given in \cite{Jo}.

When we interchange the role of the Hamiltonian function and one
of the integrals we obtain the situation closely related to the
partial integrability and the systems of Hess–-Appel'rot type
\cite{DG, Jo1, DGJ}, where the invariant manifolds are foliated on
Lagrangian tori, but not with quasi-periodic dynamics (see
Corollary \ref{partial}).

In Section 3 we consider contact flows with constraints and derive
the contact version of Dirac's construction for constrained
Hamiltonian systems (Theorem \ref{opsta}). As an example of such a
construction, we start in Section 4 from a $K$-contact sphere
$S^{2n+1}$ (see \cite{Y}) with a contact integrable Reeb flow
(Proposition \ref{sfera}). Then, in the rank 1 case, we consider
the restriction to a well known codimension 2 Brieskorn
submanifold (see \cite{Lutz, BG}). It appears that the
construction presented in Theorem \ref{opsta} provides an
alternative proof that the Brieskorn manifold is a co-oriented
contact manifold. It is well known that all the trajectories of
the corresponding Reeb flow are closed and, therefore, the system
is integrable in a noncommutative sense. Here, we describe the
Jacobi bracket within the corresponding Lie algebra of integrals
(Theorem \ref{treca}).

Further, for a class of the Brieskorn manifolds diffeomorphic to
the standard spheres $S^{4m+1}$, $m\in\mathbb N$, and contact
flows studied by Ustilovsky \cite{U}, we prove contact
noncommutative integrability of the flows with generic invariant
pre-isotropic tori of dimension $m+1$ (Proposition \ref{UB}).

Finally, we note that the integrability of the Reeb flow on a
sphere $S^{2n+1}$ is a particular case of the integrability of the
Reeb flows on compact $K$-contact manifolds (Proposition
\ref{K-contact}).

\section{Isoenergetic and partial integrability}

\subsection{Hypersurfaces of contact type}
Let $(P,\omega)$ be a symplectic $2n$-dimen\-sional manifold. Let
$H$ be a smooth function on $P$. Consider the Hamiltonian equation
\begin{equation}\label{2}
\dot x=X_H,
\end{equation}
where the Hamiltonain vector field $X_H$ is defined by
\begin{equation}\label{Hamiltonian}
i_{X_H}\omega (\,\cdot\,)=\omega(X_H,\,\cdot\,)=-dH(\,\cdot\,).
\end{equation}

The Hamiltonian $H$ is the first integral of the system. Let $M$
be a regular connected component of the invariant variety $H=h$
($dH\vert_M \ne 0$). Since $dH(\xi)=0$, $\xi\in T_x M$, from
\eqref{Hamiltonian} we see that $X_H$ generates the symplectic
orthogonal of $T_x M$ for all $x\in M$ -- the {\it characteristic
line bundle} $\mathcal L$ of $M$. It is exactly the kernel of the
form $\omega$ restricted to $M$. Note that $\mathcal L$ is
determined only by $M$ and not by $H$. If $F$ is an another
Hamiltonian defining $M$, $M\subset F^{-1}(c)$, $dF\vert_M\ne 0$,
then the restrictions of Hamiltonian vector fields $X_H$ and $X_F$
to $M$ are proportional.

An orientable hypersurface $M$ of a symplectic manifold
$(P,\omega)$ is of {\it contact type}, if there exists a 1-form
$\alpha$ on $M$ satisfying
$$
d\alpha=j^*\omega, \quad \alpha(\xi)\ne 0, \, \xi\in \mathcal L_M,
\, \xi\ne 0,
$$
where $j: M\to P$ is the inclusion (see Weinstein \cite{We}). If
$(M,\alpha)$ is of contact type, we see, owing to
$\mathcal L=\ker\omega_M$, that $\mathcal H=\{\xi\in T_x M\, \vert\, \alpha(\xi)=0, \, x\in M\}$
is a $(2n-2)$-dimensional nonintegrable distribution on which
$d\alpha=\omega$ is nondegenerate. Consequently, $\alpha\wedge
d\alpha^{n-1}$ is a volume form on $M$ and $(M,\mathcal H)$ is a
co-oriented contact manifold.

Note that, since the corresponding Reeb vector field \eqref{REEB}
is a section of $\ker d\alpha$, it is proportional to
$X_H\vert_M$. Therefore, the Reeb flow, up to a time
reparametrization, coincides with the Hamiltonian flow restricted
to $M$. In particular, a closed Reeb orbit is a closed orbit of
the Hamiltonian flow, and this was the motivation for introducing
the concept of contact type hypersurfaces \cite{We} (e.g., see
\cite{AFKP, HZ}).

In the case of an exact symplectic manifold $\omega=d\alpha$, $M$
is of contact type with respect to $\alpha$ if
$\alpha(X_H)\vert_M\ne 0$. Then $\alpha$ has no zeros in some open
neighborhood of $M$. There exists a unique vector field $E$
without zeros such that
\begin{equation}\label{liouville}
i_{E} \omega=\alpha.
\end{equation}
From Cartan's formula, \eqref{liouville} is equivalent to
$\mathcal L_E\omega=\omega, $ i.e., $E$ is the {Liouville vector
field} of $\omega$. $M$ is of contact type with respect to
$\alpha$ if and only if the Liouville vector field is transverse
to $M$, i.e., $E(H)\vert_M \ne 0$ (e.g., see Libermann and Marle
\cite{LM}). It can be shown that the Reeb vector field is given by
\begin{equation}\label{skaliranje}
Z=X_H/E(H)\vert_M.
\end{equation}

For example, the Reeb vector field $Z$ coincides with the
restrictions to $M$ of Hamiltonian vector fields of the functions
(e.g., see \cite{Jo2})
\begin{eqnarray}
&& H_0=\frac{H-h}{E(H)}  \qquad (H_0\vert_M=0),\\
&& \label{Jacobi-transf} H_{MJ}=\frac{E(H)}{4h-4H+2E(H)} \qquad
(H_{MJ}\vert_M=\frac12).
\end{eqnarray}

The transformation $H\longmapsto H_0$ is a variation of Moser's
regularization of Kepler's problem \cite{Mo, Jo2}. On the other
hand, given a natural mechanical system
$(Q,\langle\cdot,\cdot\rangle,V)$,
\begin{equation*}\label{HG}
H(q,p)=\frac12 \langle p,p\rangle +V(q)=\frac12 \sum_{ij}
K^{ij}p_ip_j+V(q),
\end{equation*}
and an isoenergetic surface $ M_h =H^{-1}(h)$, $h>\max_Q V$, the
function \eqref{Jacobi-transf}
$$
H_{MJ}(q,p)={\langle p,p\rangle }/(4(h-V(q)))
$$
is the Hamiltonian of the geodesic flow of the Maupertuis--Jacobi
metric $ds^2_{MJ}=2(h-V(q))ds^2$ on $Q$. Here $\alpha=pdq$ is the
canonical 1-form and $E=\sum_i p_i {\partial}/{\partial p_i}$  is
the {\it standard Liouville vector field} on $T^*Q(q,p)$.

\subsection{Isoenergetic integrability}
It is well known that the standard metrics on a rotational surface
and on an ellipsoid have the geodesic flows integrable by means of
an integral polynomial in momenta of the first (Clairaut) and the
second degree (Jacobi), respectively \cite{Ar}. The natural
question is the existence of metrics on a sphere $S^2$ with
polynomial integral which can't be reduced to a li\-near or a
quadratic one. The first examples, the Kovalevskaya $ds^2_K$ and
Goryachev--Chaplygin $ds^2_{GC}$ metrics with additional integrals
of 4-th and 3-rd degrees, are given by Bolsinov, Kozlov and
Fomenko (see \cite{BKF}). Namely, the motion of a rigid body about
a fixed point in the presence of the gravitation field admits
$SO(2)$--reduction (rotations about the direction of gravitational
field). Taking the integrable Kovalevskaya and
Goryachev--Chaplygin cases we get integrable systems on $T^*S^2$.
The metrics $ds^2_K$ and $ds^2_{GC}$ are then the appropriate
Maupertuis--Jacobi metrics on the sphere.

The following statement is a contact generalization of the
construction given in \cite{BKF}.

Let $\{\cdot,\cdot\}$ be the canonical Poisson bracket on a
symplectic manifold $(P,\omega)$.

\begin{prop}\label{prva}
Suppose that the Hamiltonian equations \eqref{2} are completely
integrable in a noncommutative sense with respect to the integrals
$F_1=H,\dots,F_{2n-r}$ satisfying \eqref{nehorosev}. Let
$M=H^{-1}(h)$ be a contact type hypersurface, such that the
restrictions $F_2\vert_M,\dots, F_{2n-r}\vert_M$ are independent.
Then the Reeb flow on $M$ is contact completely integrable in a
noncommutative sense with respect to the integrals
$F_2\vert_M,\dots, F_{2n-r}\vert_M$.
\end{prop}

We can say that Proposition \ref{prva} is something that one could
expect. It follows from a more general and slightly unexpected
statement.

\begin{thm}[Isoenergetic integrability]\label{druga}
Let $M=H^{-1}(h)$ be a contact type isoenergetic submanifold of
the Hamiltonian equations \eqref{2} with respect to $\alpha$:
$j^*\omega=d\alpha$, where $j: M\to P$ is the associated
inclusion. Suppose a collection of functions
$F_1=H,\dots,F_{2n-r}$ satisfy relations \eqref{nehorosev} on the
isoenergetic hypersurfce $M$ and that the restrictions
$f_2=F_2\circ j,\dots, f_{2n-r}=F_{2n-r}\circ j$ are independent.
Then the Reeb flow on $M$
\begin{equation}\label{RF}
\dot x=Z=\frac{1}{\alpha(X_H)}X_{H}
\end{equation}
is contact completely integrable in a noncommutative sense with
respect to the integrals $f_2,\dots,f_{2n-r}$. In other words, the
regular compact connected components of the invariant level sets
\begin{equation}\label{nivo}
T=T_{c_2,\dots,c_{2n-r}}: \qquad H=h,\quad F_2=c_2,\quad
\dots,\quad F_{2n-r}=c_{2n-r}
\end{equation}
are $r$--dimensional isotropic tori of $(P,\omega)$, or
pre-isotropic tori considered on $(M,\alpha)$, spanned by the
contact commuting vector fields
\begin{equation}\label{horY2}
Y_1=Z, \quad Y_k=\frac{F_k-\alpha(X_{F_k})}{\alpha(X_H)}
X_{H}+X_{F_k}, \qquad k=2,\dots,r.
\end{equation}
\end{thm}

Note that the restriction of the Hamiltonian flow \eqref{2} to $M$
and  the Reeb flow \eqref{RF} are related by the time
reparametrization $dt=\alpha(X_H)d\tau$.

\medskip

\begin{proof}
As in the proof of the usual noncommutative integrability, from
$$
X_{F_i}(F_j)=\{F_j,F_i\}=0\vert_M, \qquad i=1,\dots,r, \quad
j=1,\dots, 2n-r,
$$
we have that $X_{F_i}$, $i=1,\dots,r$ are tangent to the invariant
sets \eqref{nivo}. Also, if $T$ is regular component of
\eqref{nivo}, from the dimensional reasons, its tangent space is
spanned with $X_{F_i}$, $i=1,\dots,r$. This is an isotropic
manifold since
$$
\omega(X_{F_i},X_{F_j})=\{F_j,F_i\}=0\vert_M, \quad i,j=1,\dots,r.
$$

It is a nontrivial fact that $T$ admits a transitive $\R^r$-action
and therefore it is diffeomorphic to a $r$-dimensional torus.
Namely, conditions \eqref{nehorosev} on $M$  imply
\begin{equation*}
[X_{F_k},X_{F_i}]=\sigma_{ki}X_{F_1}=\sigma_{ki} X_H, \quad
k=1,\dots,r, \quad i=1,\dots,2n-r,
\end{equation*}
for certain functions $\sigma_{ki}$ defined on $M$. Thus, $T$ is a
torus if one can deform tangent vector fields
$X_{F_1},\dots,X_{F_r}$ to commuting independent vector fields on
$T$. This, together with a noncommutative integrability of the
Reeb flow \eqref{RF} follows from the consideration below.

To avoid a confusion between the objects defined on $M$ and $P$,
in what follows by $\tilde X$ we denote the restrictions of the
vector fields $X$ tangent to $M$: $j_*(\tilde X)=X$.

First note, since $Z$ is proportional to $\tilde X_H$, the
restrictions $f_2,\dots,f_{2n-r}$ are integrals of the Reeb flow
\eqref{RF}. That is, $df_i$ are semi-basic forms on $M$, or in
terms of the Jacobi bracket \eqref{der}:
$$
[1,f_i]=0, \qquad i=2,\dots,2n-r.
$$
Thus, we need to prove
$$
[f_k,f_i]=0, \qquad i=2,\dots,2n-r, \quad k=2,\dots,r,
$$
which is equivalent to the commuting of the associated contact
Hamiltonian vector fields: $ [Y_{f_k},Y_{f_i}]=0$.

Recall that a horizontal part $\hat{Y}_{f_i}$ of $Y_{f_i}$
($\alpha(\hat{Y}_{f_i})=0$) is defined by
$\hat{Y}_{f_i}=\alpha^\sharp(\hat{df_i})$, and since $df_i$ is
semi-basic, we have
\begin{equation}\label{4}
i_{\hat{Y}_{f_i}} d\alpha=-\hat{df_i}=-df_i.
\end{equation}
On the other hand, the Hamiltonian vector field $X_{F_i}$ is
tangent to $M$ and satisfies
\begin{equation}\label{5}
i_{X_{F_i}} \omega=-dF_i.
\end{equation}
By taking the pull back of \eqref{5} and combining with \eqref{4}
we get
$$
\hat{Y}_{f_i}-\tilde X_{F_i}\subset \ker d\alpha.
$$
Whence,
\begin{equation*}\label{horY}
\hat{Y}_{f_i}=\tilde X_{F_i}-\alpha(\tilde X_{F_i})Z,
\end{equation*}
and the contact Hamiltonian vector fields of the functions $f_i$
are given by
$$
Y_{f_i}=f_i Z + \tilde X_{F_i}-\alpha(\tilde X_{F_i})Z, \qquad
i=2,\dots,2n-r.
$$

By definition of the Jacobi bracket \eqref{der} we have:
\begin{eqnarray*}
[f_k,f_i] &=& {Y_{f_k}}(f_i)- f_i Z(f_k)\\
&=&Y_{f_k}(f_i)\\
&=&(f_k Z+\hat{Y}_{f_k})(f_i)\\
&=& \tilde X_{F_k} (f_i)\\
&=&j^* (X_{F_k} (F_i))\\
&=&j^* \{F_i,F_k\}=0,
\end{eqnarray*}
for $i=2,\dots,2n-r, \, k=2,\dots,r.$
\end{proof}

In particular, for $r=n$ we have the isoenergetic version of the
Arnold-Liouville theorem:

\begin{cor}[Isoenergetic Arnold-Liouville
theorem]\label{IAL} Let $M=H^{-1}(h)$ be a contact type
hypersurface. Suppose a collection of functions
$F_1=H,\dots,F_{n}$ Poisson commute on $M$
\begin{equation}\label{komutiranje}
\{F_i,F_j\}=0\vert_M
\end{equation}
and that the restrictions $F_2\vert_M,\dots, F_{n}\vert_M$ are
independent. Then the Reeb flow on $M$ is contact completely
integrable with respect to the integrals $F_2\vert_M,\dots,
F_{n}\vert_M$.
\end{cor}

Thus, the system doesn't need to be integrable on the whole phase
space, while the isoenergetic hypersurface $M$ is foliated on
Lagrangian tori, or pre-Legendrian in the contact sense. In a
neighborhood of a torus, there exist contact action-angle
coordinates $(\theta_0,\dots,\theta_{n-1},y_1,\dots,y_{n-1})$,
such that $y_i$ depends on the integrals
$F_2\vert_M,\dots,F_n\vert_M$, the contact form has the canonical
form $
y_0(y_1,\dots,y_{n-1})d\theta_0+y_1d\theta_1+\dots+y_{n-1}d\theta_{n-1}$ in which the Reeb vector field is linearized
$$
Z=\sum_{i=0}^{n-1} \omega_i
\frac{\partial}{\partial\theta_i},\qquad
\omega_0=\left(-y_0+\sum_{i=1}^{n-1}y_i\frac{\partial
y_0}{\partial y_i}\right)^{-1}, \qquad
\omega_i=-\frac{1}{\omega_0}\frac{\partial y_0}{\partial y_i}.
$$

Although the problem of existence of polynomial in momenta first
integrals for natural mechanical systems on a fixed isoenergetic
hypersurface is well known and goes back to Birkhoff \cite{Bir0,
Bir} (the examples are given by Yehia, see \cite{Ye} and
references therein), the formulation of the isoenergetic
integrability, to the authors knowledge, has not been given yet.
Furthermore, in the case of the natural mechanical systems, a
compact regular component $M$ of the isoenergetic hypersurface
$H^{-1}(h)$ is always of contact type. This is obvious, when
$h>\max_Q V$. In general, if $h < \max_Q V$ we can perturb the
canonical 1-form $pdq$ by a closed 1-form form $\beta$, such that
$M$ is of contact type with respect to $pdq+\beta$ (e.g., see
\cite{HZ}).

Another approach to the integrability on invariant manifolds is
given by Nekhoroshev \cite{N2, N3}.

\subsection{Partial integrability}
Next, we can interchange the role of the Hamiltonian function and
one of the integrals in Corollary \ref{IAL}.

\begin{cor}\label{partial}
Suppose that a Hamiltonian system \eqref{2} has $n-1$ commuting
integrals $F_1=H,F_2,\dots,F_{n-1}$ and an invariant relation
$$
\Sigma: \qquad F_0=0,
$$
that is, the trajectories with initial conditions on $\Sigma$ stay
on $\Sigma$ for all time $t$. If $\Sigma$ is of the
contact type manifold and if it is invariant for all Hamiltonian
flows $X_{F_i}$, then the compact regular components of the
invariant varieties
$$
F_0=0, \quad H=F_1=c_1, \quad F_2=c_2, \quad \dots, \quad
F_{n-1}=c_{n-1}
$$
are Lagrangian tori.
\end{cor}

The situation is closely related to the notion of the systems of
Hess–-Appel'rot type introduced by Dragovi\'c and Gaji\' c
\cite{DG} as well as the notion of partial integrability given in
\cite{Jo1}. Although $\Sigma$ is foliated on invariant Lagrangian
tori, the dynamics of \eqref{2} does not need to be solvable.

\subsection{Remark}
If we are interested in the dynamics in the invariant neighborhood
$U\subset M$ of the regular compact invariant level set $T$ (with
the property $dF_1\wedge\dots\wedge dF_{2n-r}\ne 0\vert_U$), then
instead of the condition that $M$ is of contact type in Theorem
\ref{druga}, we can assume a slightly weaker condition that $U$ is
of contact type. Then we have contact integrability of the Reeb
flow restricted to $U$. Similarly, the invariant relation $\Sigma$
in Corollary \ref{partial} does not need to be of contact type.

\subsection{Example}
A classical example of a partially integrable system is the
Hess–-Appel'rot case of the heavy rigid body motion around a fixed
point. The phase space is the cotangent bundle of $SO(3)$. In
Euler's angles $(\varphi,\theta,\psi)$, the Hamiltonian of the
system can be written in the form
$$
H=\frac12(a M_1^2+ a M_2^2+ b M_3^2+2c M_1 M_3)+k\cos\theta,
$$
where the components of the angular momentum are
\begin{eqnarray*}
&&
M_1=\frac{\sin\varphi}{\sin\theta}(p_\psi-p_\varphi\cos\theta)+p_\theta\cos\varphi,\\
&&
M_2=\frac{\cos\varphi}{\sin\theta}(p_\psi-p_\varphi\cos\theta)-p_\theta\sin\varphi,
\qquad M_3=p_\varphi.
\end{eqnarray*}

The system has two integrals, the Hamiltonian function $H$ and the
Noether integral $M_z=p_\psi$ corresponding to the symmetry of the
system with respect to the rotations around the vertical axes
$\vec e_z$. Also, the system has the invariant relation
$$
\Sigma: \qquad M_3=p_\varphi=0.
$$

Since the Hamiltonian flow of the function $M_3$ is given by the
vector field $ X_{M_3}={\partial}/{\partial \varphi}$ (rotations
of the body around the axes $\vec e_3$ fixed in the body), we see
that $\Sigma$ is not of the contact type with respect to the
canonical 1-form $p_\varphi d\varphi+p_\theta d\theta+p_\psi
d\psi$, but it is of the contact type with respect to the
perturbation of the canonical 1-form by a closed 1-form
$d\varphi$:
$$
\alpha=p_\varphi d\varphi+p_\theta d\theta+p_\psi d\psi+d\varphi,
\qquad \alpha(X_{M_3})\equiv 1.
$$

Further, since
$$
\{M_3,M_z\}=0, \qquad \{p_\varphi,p_\psi\}=0,
$$
from Corollary \ref{partial} we obtain that the regular connected
invariant level sets
$$
H=h, \qquad M_z=c, \qquad M_3=0,
$$
are Lagrangian tori. This is well known and can be seen directly
from the fact that, after confining to $\Sigma$, the Hamiltonian
of the Hess-Appel'rot case coincides with the Hamiltonian of the
integrable Lagrange case of the heavy rigid body motion. However,
the dynamics on $\Sigma$ of the Hess-Appel'rot system is quite
different from the one of the Lagrange top. For a complete
integration one should solve an additional Riccati equation (e.g.,
see \cite{G} for both: the classical and the algebro-geometric
integration of the system).

\subsection{Example} As an illustration of Proposition \ref{prva},
we give the example of an integrable contact flow, which is not a
a geodesic flow of the Maupertuis--Jacobi metric. Consider the
simplest integrable system on the standard linear symplectic space
$(\R^{2n+2}(q,p), dp\wedge dq)$, the system of $n+1$ harmonic
oscillators with the Hamiltonian
\begin{equation}\label{H0}
H_0=\sum_{i=0}^{n}\frac{1}{2a_i}(q_i^2+p_i^2),
\end{equation}
where $a_i$ are positive numbers. The commuting integrals are
$F_i=q_i^2+p_i^2$, $i=0,\dots,n$. If instead of the canonical
1-form $pdq$ we take
\begin{equation}\label{pert}
\alpha_0=\sum_{i=0}^{n} p_idq_i-\frac12d(\sum_{i=0}^{n} p_i
q_i)=\frac12\sum_i p_i dq_i-q_i dp_i,
\end{equation}
then $d\alpha_0=d(pdq)=dp\wedge dq$ and the only zero of
$\alpha_0$ is at the origin $0$. The corresponding Liouville
vector field is
$$
E=\frac12\sum_i q_i\frac{\partial}{\partial q_i}+p_i
\frac{\partial}{\partial p_i}.
$$

Let $E_h$ be an ellipsoid $H_0=h$, $h>0$. We have
$E(H_0)\vert_{E_h}=h\ne 0$. Therefore, $E_h$ is of contact type
with respect to $\alpha_0$. According to \eqref{skaliranje}, the
Reeb vector field on $(E_h,\alpha_0)$ is given by
$Z=\frac{1}{h}X_{H_0}$. The contact ellipsoid $(E_h,\alpha_0)$ is
contactomorphic to the $K$-contact structure on a sphere
$S^{2n+1}$ \cite{Y} (see Section 4). In particular, for
$a_0=a_1=\dots=a_{n}=1$ we get the standard contact structure on
the sphere $S^{2n+1}=H_0^{-1}(h)$, where the characteristic line
bundle defines the Hopf fibration.

\begin{lem}\label{U(n)dejstvo}
The standard linear action of $U(n+1)$ on $\R^{2n+2}\cong \mathbb
C^{n+1}$ ($p+ iq=z$) is Hamiltonian with the momentum mapping
$$
\Phi: \R^{2n+2}\to \mathfrak u(n+1)\cong \mathfrak u(n+1)^*, \quad
\Phi(q,p)=p\wedge q+ i(q\otimes q+p\oplus p).
$$
\end{lem}

Here $\mathfrak u(n+1)$ is identified with $\mathfrak u(n+1)^*$ by
the use of the product proportional to the Killing form: $\langle
\xi,\eta\rangle=-\frac12\tr(\xi\eta)$.

The Hamiltonian \eqref{H0} for $a_0=a_1=\dots=a_{n}=1$
 is a collective function -  the pull-back of a
linear Casimir function
$$
K_0(\xi)=-\frac12\tr(\xi \zeta), \qquad \zeta=\diag(i,\dots,i)
$$
via the momentum mapping $\Phi$. Note that $\Phi(\R^{2n+2})$ is
the union of singular adjoint orbits. For $(q,p)\ne 0$, the orbit
through $\Phi(q,p)$ is diffeomorphic to a complex projective space
$\mathbb{CP}^n\cong U(n+1)/U(n)\times U(1)$. On the other hand,
the orbits of $U(n+1)$-action on $\R^{2n+2}$ for $(q,p)\ne 0$ are
the spheres $H_0^{-1}(h)$.

Let us take an arbitrary integrable system on $\mathfrak u(n+1)$.
For example, consider the Gelfand--Cetlin system that is defined
by the filtration of Lie algebras  $ \mathfrak u(1)\subset
\mathfrak u(2)\subset \dots \subset \mathfrak u(n+1)$ (e.g., see
\cite{GS}).

Let $\xi_{\mathfrak u(k)}$ be the projection of $\xi$ to
$\mathfrak u(k)$ with respect to the scalar product $\langle
\cdot,\cdot\rangle$. The Euler equation $ \dot \xi=[\xi,\nabla
K(\xi)] $ with the Hamiltonians
$$ K(\xi)=-\frac12\sum_{k=0}^n
{\lambda_k}\tr(\xi_{\mathfrak u(k+1)} \xi_{\mathfrak u(k+1)}), $$
where $\lambda_k$ are real parameters, are completely integrable.

Note that the $U(n+1)$-action is multiplicity free, so the system
with collective Hamiltonian $H=K\circ\Phi$ is also integrable and
we need only Noether integrals for complete integrability
\cite{GS}. Indeed, we have $H=\sum_{k=0}^n \lambda_k H_k(q,p)$,
where
$$
H_k=\frac12\sum_{i=0}^k (q_i^2+p_i^2)^2+\sum_{0\le i<j\le r}
(q_iq_j+p_ip_j)^2+(q_jp_i-p_jq_i)^2, \quad k=0,\dots,n
$$
are commuting integrals of the system.

Let $\Sigma: \, H=h\ne0$. It is an algebraic surface of degree 4.
Since $E(H)=2h\vert_{\Sigma}$, $\Sigma$ is of contact type
 and the Reeb flow is
$Z=(2h)^{-1}X_H\vert_{\Sigma}$.

We can summarize the considerations above in the following
statement.

\begin{prop}
The Reeb flow on $\Sigma$ is completely integrable.
\end{prop}

Similarly, let a compact connected Lie group $G$ acts in a
Hamiltonian way on a symplectic manifold $(P,\omega)$ with the
equivariant momentum mapping $\Phi: P\to \mathfrak g^*$ (doesn't need
to be multiplicity free action). Let $K: \mathfrak
g^*\to\mathbb{R}$ be a Hamiltonian function such that the Euler
equations
\begin{equation} \label{Euler}
\dot \mu=-\ad^*_{dK(\mu)}\mu
\end{equation}
are completely integrable on general co-adjoint orbits ${\mathcal
O}(\mu)\subset\Phi(M)$ with a set of Lie--Poisson commuting
integrals $f_i$, $i=1,\dots,N$, $N=\dim\mathcal O(\mu)$. Then the
Hamiltonian equations on $P$ with a collective Hamiltonian
function $H=K\circ\Phi$ are completely integrable (in the
non-commutative sense). The complete set of first integrals is
$$
\{f_i\circ\Phi\mid i=1,\dots,N\}+C^\infty_G(P),
$$
where $C^\infty_G(P)$ is the algebra of $G$-invariant functions
\cite{BJ}. Thus, if a isoenergetic hypersurface $M=H^{-1}(h)$ is
of a contact type, the associated Reeb vector field will be
completely integrable.

Recently, as an application of Proposition \ref{prva}, a discrete
Hamiltonian system, namely Heisenberg model on a product of
light--like cones in a pseudo--Euclidean space, which induces an
integrable contact transformation on certain contact hypersurfaces
is given in \cite{Jo3}.

\section{Contact systems with constraints}

A {\it contact submanifold} of the contact manifold $(M,{\mathcal
H}_M)$ is a triple $(N,{\mathcal H}_N, j)$, where $(N,{\mathcal
H}_N)$ is a contact manifold and $j: N\rightarrow M$ is an
embedding such that $j_{\ast}^{-1}({\mathcal H}_M)={\mathcal
H}_N$.

Let $(M,\alpha)$ be a co-oriented contact manifold and $j:
N\rightarrow M$ an embedding. If we define
\[
{\mathcal H}_N=\{X\in TN\, \vert\, j_{\ast}(X)\in {\mathcal
H}_M\}=j_{\ast}^{-1}({\mathcal H}_M),
\]
then  ${\mathcal H}_N=\ker(
j^{\ast}\alpha)$. The distribution ${\mathcal H}_N$ is of codimension one, if $N$ is transverse to ${\mathcal H}_M$.
In order to induce a contact structure on $N$ it
is also necessary that $d j^{\ast}\alpha$ is non-degenerate on
$({\mathcal H}_N)_x$, for all $x\in N$. To summarize,
$(N,j^{\ast}\alpha)$ is a {\it contact co-oriented submanifold} of
$(M,\alpha)$, if $N$ is transverse to ${\mathcal H}_M$ and if
$dj^{\ast}\alpha$ is non-degenerate on $\ker(j^{\ast}\alpha)$.

We derive a contact version of Dirac's construction which deals
with constrained Hamiltonian vector fields on symplectic manifolds
(e.g., see \cite{MZ}).

\begin{thm}\label{opsta}
Let $(M,\alpha)$ be a $(2n+1)$-dimensional co-oriented contact
manifold, $G_1,\dotso,G_{2k}$ smooth functions on $M$,
\begin{equation}\label{NG}
N=\{x\in M\,\vert\, G_1(x)=\dotso= G_{2k}(x)=0\},
\end{equation}
and $j: N\rightarrow M$ be the corresponding embedding.

\

\noindent {\rm (a)} If  $[1,G_j]=0\vert_N$, $j=1,\dotso,2k$ and
\begin{equation}\label{G}
\det([G_j,G_i])\neq0\vert_N
\end{equation}
then $(N,j^{\ast}\alpha)$ is a contact submanifold of $(M,\alpha)$
with the Reeb vector field that is the restriction of the the Reeb
vector field $Z$ of $(M,\alpha)$.

\

 \noindent {\rm (b)} Let
$f$ be a smooth function on $M$ and
$$
W_f=Y_f-\sum_{i=1}^{2k}\lambda_i Y_{G_i}.
$$
Then the system
\begin{equation}\label{nondeg}
dG_j(W_f)=Y_f(G_j)-\sum_i\lambda_i Y_{G_i}(G_j)=0\qquad
j=1,\dotso,2k
\end{equation}
has a unique solution
$\lambda_1=\lambda_1(f),\dotso,\lambda_{2k}=\lambda_{2k}(f)$ on
$N$. For the given multipliers,
$$
Y^*_f=W_f
$$
is the contact Hamiltonian  vector field of the function $f$
restricted to $N$. If $g$ is any smooth function on $M$, the
Jacobi bracket between the restrictions of $f$ and $g$ to $N$ is
given by
\begin{equation}\label{dirac}
[f\vert_N,g\vert_N]_N=[f,g]+\sum_{i,j} [G_i,g]A_{ij}[G_j,f],
\end{equation}
where $A_{ij}$ is the inverse of the matrix $([G_i,G_j])$. In
particular, if either $Y_f$ or $Y_g$ is tangent to $N$, then
$[f\vert_N,g\vert_N]_N=[f,g]$.
\end{thm}

\begin{proof}
(a) \enspace From the conditions $[1,G_j]=0\vert_N$,
$j=1,\dots,2k$ and \eqref{G}, it easily follows that
$dG_1,\dotso,dG_{2k}$ are linearly independent semi-basic forms on
$N$, and that the Reeb vector field $Z$ is tangent to $N$. Hence
$N$ is a submanifold transverse to ${\mathcal H}_M$.

Let
$$
\mathcal H_N=\{\xi\in T_xN\,\vert\, \alpha(\xi)=0\}.
$$

We need to prove that $d\alpha$ is non-degenerate on $\mathcal
H_N$, i.e. that $\mathcal H_N$ is a symplectic subbundle of
${\mathcal H}_M$. Owing to $\dim \mathcal H_M=2n$, $\dim\mathcal
H_N=2n-2k$, we obtain that the dimension of the symplectic
orthogonal to $\mathcal H_N$ within $\mathcal H_M$ equals $2k$.

Since $dG_1,\dots,dG_{2k}$ are independent semi-basic forms on
$N$, we have that $Z, Y_{G_1},\dotso,Y_{G_{2k}}$, i.e.,  $\hat
Y_{G_1},\dotso,\hat Y_{G_{2k}}$ are linearly independent vector
fields on $N$. Whence, from
\[
 d\alpha(\hat Y_{G_j},
X)=-dG_j(X)=0\vert_N, \qquad X\in\mathcal H_N,
\]
we see that $\orth_{\mathcal H_M} \mathcal H_N$ is spanned by
$\hat Y_{G_1},\dotso,\hat Y_{G_{2k}}$. The relations
$$
d\alpha(\hat Y_{G_j},\hat Y_{G_i})=d\alpha(Y_{G_j}, Y_{G_i})=[G_j,G_i]
$$
and \eqref{G} yield that $\orth_{\mathcal H_M} \mathcal H_N$ is
symplectic, hence $\mathcal H_N$ is symplectic as well.

Since  $Z$ is tangent to $N$, it is obvious that $Z\vert_N$ is the
Reeb vector field  of $(N,j^{\ast}\alpha)$. Also,  note that we
have the following decompositions of $T_x M$ at $x\in N$:
\begin{equation}\label{GG}
T_x M=\mathcal Z_x \oplus {\mathcal H_M}_x=\mathcal Z_x\oplus
{\mathcal H_N}_x\oplus {\orth_{\mathcal H_M}\mathcal H_N}_x=T_x
N\oplus {\orth_{\mathcal H_M}\mathcal H_N}_x.
\end{equation}

(b)\enspace   According to \eqref{der}, we can write the equations
for multipliers \eqref{nondeg} in terms of the Jacobi brackets
\begin{equation}\label{dirac2}
\sum_i \lambda_i \,[G_i,G_j]=[f,G_j], \qquad j=1,\dots,2k,
\end{equation}
where we used $G_j\vert_N \equiv 0$, $j=1,\dots,2k$.

Thus, \eqref{nondeg} has a unique solution on $N$,
$$
\lambda_i(f)=\sum_j A_{ij} [f,G_j],
$$
 determining the
projection $W_f$ of $Y_f$ to $TN$ with respect \eqref{GG}. From
\[
i_{W_f}d\alpha(\xi)=-df(\xi)+\ds\sum_{j=1}^{2k}\lambda_jdG_j(\xi)=-df(\xi),
\qquad \xi\in T_x N
\]
we obtain that the contact Hamiltonian vector field of $f\vert_N$
reads
\begin{eqnarray*}
Y^*_f &=& \hat{W}_f+fZ\\
&=& Y_f-\sum_i\lambda_i(f) Y_{G_i} - (\alpha(Y_f)-\sum_i
\alpha(\lambda_i(f)Y_{G_i}))Z+fZ\\
&=& Y_f-\sum_i \lambda_i(f) Y_{G_i} -(f-\sum_i\lambda_i(f) G_i)Z+fZ\\
&=& Y_f-\sum_i \lambda_i (f)Y_{G_i}.
\end{eqnarray*}

Finally, let $g$ be a smooth function on $M$. Then
\begin{eqnarray*}
[f\vert_N,g\vert_N]_N &=& Y^*_f(g)-g Z(f)\\
&=& Y_f(g)-g Z(f)-\sum_i \lambda_i(f) Y_{G_i}(g)\\
&=& [f,g]-\sum_i \lambda_i (f)[G_i,g]\\
&=& [f,g]+\sum_{i,j}[G_i,g]A_{ij}[G_j,f].
\end{eqnarray*}
\end{proof}

Now, let us consider a general situation, when the Reeb vector
field $Z$ of $(M,\alpha)$ is not tangent to the contact co-oriented
submanifold $(N,j^{\ast}\alpha)$. Then, analogous to \eqref{GG},
we have the decompositions of $T_x M$ at the points $x\in N$:
\begin{eqnarray}
T_x M &=&\mathcal Z_x \oplus {\mathcal H_M}_x=\mathcal Z_x\oplus
{\mathcal H_N}_x\oplus {\orth_{\mathcal H_M}\mathcal H_N}_x \nonumber \\
&=& \mathcal Z^*_x\oplus  {\mathcal H_N}_x   \oplus
{\orth_{\mathcal H_M}\mathcal H_N}_x=T_x N\oplus {\orth_{\mathcal
H_M}\mathcal H_N}_x,\label{GGG}
\end{eqnarray}
where $\mathcal Z^*$ is a linear subbundle of $TN$ spanned by the
Reeb vector field $Z^*$ of $(N,j^{\ast}\alpha)$. For a given
smooth function $f$ on $M$, the contact Hamiltonian vector field
of $f\vert_N$ reads
$$
Y_f^*=\hat{W}_f+fZ^*=W_f-\alpha(W_f)Z+fZ^*,
$$
where $W_f$ is the projection to $TN$ of $Y_f$ with respect to the
decomposition \eqref{GGG}.

If $N$ is given by \eqref{NG}, then we have the same condition
\eqref{G} on the constraints $G_1,\dots,G_{2k}$ and the same
equations \eqref{dirac2} determining the multipliers
$\lambda_1,\dots,\lambda_{2k}$, while the expression for the
Jacobi bracket \eqref{dirac} is different:
\begin{eqnarray*}
&& [f\vert_N,g\vert_N]_N = Y^*_f(g)-g Z^*(f)\\
&&\qquad = Y_f(g)-\sum_i \lambda_i(f) Y_{G_i}(g)-(f-\sum_i
\lambda_i(f)
{G_i})Z(g)+f Z^*(g)-g Z^*(f)\\
&&\qquad= [f,g]+gZ(f)-fZ(g)+f Z^*(g)-g Z^*(f)-\sum_i \lambda_i(f)
Y_{G_i}(g).
\end{eqnarray*}

\subsection{Integrability of the Reeb flows} Suppose that the Reeb flow
$$
\dot x=Y_1=Z
$$
on a contact manifold $(M^{2n+1},\alpha)$ is integrable by means
of integrals $f_1,f_2,\dots,f_{2n-r}$ satisfying
\eqref{involucija} and let $M_{reg}$ denote an open set foliated on
invariant pre-isotropic tori of dimension $r+1$. In addition,
assume that the system is subjected to the constraints \eqref{NG}
satisfying the conditions of Theorem \ref{opsta}.

Since the Reeb vector field on $N$ is the restriction $Z\vert_N$,
it is clear that, in the case $M_{reg} \cap N$ is an open dense
set of $N$, a generic Reeb trajectory on $N$ will be
quasi-periodic. However, as the example with the Reeb flow on the
Brieskorn manifold suggests (see the next section), the restricted
flow not need to be integrable by means of integrals obtained by
the restrictions of $f_1,f_2,\dots,f_{2n-r}$ to $N$.

Denote the foliation of $M_{reg}$ on invariant pre-isotropic tori
by $\mathcal F$. Then $\mathcal F$ is a $\alpha$-complete
pre-isotropic foliation and we have the following flag of
foliations
\begin{equation}\label{flag}
\mathcal G=\mathcal F \cap \mathcal H_M \,\,\subset\,\, \mathcal
F=\mathcal Z\oplus \mathcal G \,\,\subset \,\,\mathcal E=\mathcal
Z\oplus \orth_{\mathcal H_M}\mathcal G
\end{equation}

Let $\mathcal F_N$ be the foliation given by the level sets
$f_i(x)=c_i$, $i=1,\dots,2n-r$, $x\in N_{reg}$, that is $\mathcal
F_N=\mathcal F\vert_N \cap TN$ (if necessary, in order to obtain a
foliation, we restrict functions to some open dense set of
$N_{reg}$). Further, let $\mathcal G_N=\mathcal F_N\cap\mathcal
H_N=\mathcal G\vert_N \cap\mathcal H_N$ be an isotropic foliation
of $N_{reg}$. Then the Reeb flow is integrable by means of
integrals $f_1\vert_N,\dots,f_{2n-r}\vert_N$ if the distribution
$$
{\mathcal E_N}=\mathcal Z\oplus \orth_{\mathcal H_N}{\mathcal G_N}
$$
is a foliation (see \cite{Jo}). We have
\begin{eqnarray*}
{\mathcal E_N}&=& \mathcal Z\oplus \orth_{\mathcal H_N}{\mathcal
G_N}=\mathcal Z\oplus (\orth_{\mathcal H_M}({\mathcal
G}\cap{\mathcal H_N})\cap{\mathcal H_N})\\
&=& \mathcal Z \oplus
((\orth_{\mathcal H_M}{\mathcal G}+\orth_{\mathcal H_M}{\mathcal H_N})\cap \mathcal H_N)\\
&=& \mathcal Z \oplus \pr_{\mathcal H_N}(\orth_{\mathcal
H_M}\mathcal G)= \pr_{TN}(\mathcal Z\oplus \orth_{\mathcal
H_M}\mathcal G)=\pr_{TN}\mathcal E,
\end{eqnarray*}
where $\pr_{\mathcal H_N}$ and $\pr_{TN}$ denote the projections
with respect to the decompositions of $TM$ given in \eqref{GG}.

In particular, if $\mathcal F\vert_N$ is tangent to $N$
(${\mathcal F_N}=\mathcal F\vert_N$), then $\mathcal F_N$ is
$\alpha$-complete preisotropic foliation and the Reeb flow on $N$
is integrable. Indeed, in this case $\mathcal E_N$ is an
integrable distribution, since $\mathcal G\vert_N\subset \mathcal
H_N$ implies
$$
\pr_{\mathcal H_N}(\orth_{\mathcal H_M}\mathcal G)=\orth_{\mathcal
H_M}\mathcal G \cap \mathcal H_N \qquad \text{and}\qquad
\pr_{TN}\mathcal E =\mathcal E\cap TN.
$$

For example, if the generic Reeb trajectories in $N_{reg}$ are
everywhere dense over the $r+1$-dimensional isotropic tori then it
is obvious that ${\mathcal F_N}=\mathcal F\vert_N$.

\section{Contact flows on Brieskorn manifolds}

In Section 2 we started from a system of identical harmonic
oscillators and gave a simple construction of a contact
hypersurface in $(\R^{2n+2}(q,p), dp\wedge dq)$ with an integrable
Reeb flow associated to the Gelfand--Cetin system on $u(n+1)$.

In this section we are going to a different direction. We also
start from a system of harmonic oscillators \eqref{H0}, but then
change our contact manifold by imposing certain constraints thus obtaining the contact structure on the Brieskorn manifold \cite{Lutz, BG}.

\subsection{Complete contact integrability on a sphere}

In this section we use the identification
$\C^{n+1}\cong\R^{2n+2}$,  $z_j=x_j+iy_j,\enspace (j=0,\dots, n)$.
Let
$$
F(z,\bar{z})=\sum_{j=0}^n|z_j|^2.
$$
Then the unit sphere $S^{2n+1}$ may be expressed as the level
surface $F=1$.

In \cite{Y}, the following basic examples of $K$-contact manifolds
are given:

\begin{prop}
Let $a_j\enspace (j=0,\dots, n)$ be positive real numbers and
$$
\alpha =\ds\frac{i}8\sum_{j=0}^n a_j(z_jd\zkj-\zkj
dz_j)=\frac14\sum_{j=0}^na_j(x_j dy_j-y_j dx_j).
$$
Then $(S^{2n+1},\,\alpha)$ is a co-oriented contact manifold with
the Reeb vector field
\begin{equation}\label{Reeb}
Z=4i\sum_{j=0}^n\frac1{a_j}\left(z_j\dzj-\zkj\dzkj\right).
\end{equation}
\end{prop}

Note that the coordinate transformation
\begin{equation}\label{promena}
p_i=\sqrt{\frac{a_i}{2}}x_i, \quad q_i=\sqrt{\frac{a_i}{2}}y_i,
\quad i=0,\dots,n
\end{equation}
transforms the contact ellipsoid $(E_{1/4},\alpha_0)$ to the
contact sphere $(S^{2n+1},\alpha)$. Now, the contact integrability
of the Reeb flow on $(S^{2n+1},\alpha)$ directly follows from
Proposition \ref{prva}. However, it is instructive to have a
direct proof as well.

\begin{prop}\label{sfera}
The flow induced by the Reeb vector field \eqref{Reeb} is completely contact integrable on $(S^{2n+1},\alpha)$ and the functions
\begin{equation}\label{integr}
f_j(z)=|z_j|^2,\qquad j=0,\dots,n
\end{equation}
are its commuting integrals.
\end{prop}

\begin{proof}
Let
\begin{equation*}
Y_j=\frac{4i}{a_j}\left(z_j\dzj-\zkj\dzkj\right), \qquad
j=0,\dots,n.
\end{equation*}

From
$$
Z f_j=df_j(Z)=(\zkj dz_j+z_jd\zkj)(Z)=\frac{4i}{a_j}(\zkj
z_j-z_j\zkj)=0
$$
and $dF(Y_j)=\frac{4i}{a_j}(\zkj
z_j-z_j\zkj)=0$, we conclude that $f_j$ are integrals of
\eqref{Reeb} and that $Y_j$ are tangent to $S^{2n+1}$. Since
$$
i_{Y_j}d\alpha=-(z_jd\zkj+\zkj dz_j)=-df_j
$$
and $\alpha(Y_j)=f_j$, it follows that $Y_j$ are the contact
Hamiltonian vector fields of the functions $f_j$. Whence, from
\eqref{der}, we obtain
$$
[f_j,f_k]=Y_jf_k=df_k(Y_j)=0, \qquad j,k=0,\dots,n.
$$

Obviously, the functions $f_j$, $j=0,\dots,n$ are independent on
the open dense subset $
 U=\{(z_0,\dots, z_n)\,\vert\,
z_0\cdot\dotso\cdot z_n\neq0\}$ of $\C^{n+1}$,
while the  restrictions of $f_j$, $j=1,\dots,n$ to $S^{2n+1}$ are
independent on $S^{2n+1}\cap U$. Formally, in the polar
coordinates
\begin{equation}\label{polar}
z_j=r_j e^{i\varphi_j}, \qquad j=0,\dots,n,
\end{equation}
we have
$
dF\wedge df_1\wedge\dots\wedge df_n=2^{n+1}r_0r_1\dots r_n
dr_0\wedge dr_1 \wedge \dots \wedge dr_n  \ne 0\vert_U.
$
Therefore, $ df_1\wedge\dotso\wedge df_n\neq0\vert_{S^{2n+1}\cap
U}, $ considered as a $n$-form on $S^{2n+1}$.
\end{proof}

Consider a  $(n+1)$-dimensional invariant torus
\begin{equation}\label{torusi}
 T_c: \quad f_0=c^2_0, \quad \dots, \quad
f_{n}=c^2_{n}, \qquad c^2_0+c^2_1+\dots+c^2_n=1
\end{equation}
that lays in $S^{2n+1}\cap U$ ($c_0,\dots,c_n>0$). After a
coordinate change \eqref{polar}, the Reeb dynamics on $T_c$ is
given by
\begin{equation*}
\dot\varphi_j=\omega_j=\frac{4}{a_j}, \qquad j=0,\dots,n.
\end{equation*}

If $\omega_0,\omega_1,\dots,\omega_n$ (that is,
$a_0,a_1,\dots,a_n$) are independent over $\mathbb Q$, then the
trajectories are everywhere dense over $T_c$. In general, if among
$\omega_0,\omega_1,\dots,\omega_n$ we have $r$ independent
relations
$$
\rho_k: \qquad m^k_0 \omega_0+\dots+m^k_n \omega_n=0, \qquad
m^k_i\in\mathbb Z, \qquad k=1,\dots,r,
$$
the dimension $n+1-r$ of the closure of trajectories laying on
$T_c$ is equal to the rank of $(S^{2n+1},\alpha)$ considered as a
$K$-contact manifold \cite{Y}.

\subsection{Reduction to Brieskorn manifolds}

Now, let $a_j\enspace (j=0,\dots, n)$ be positive integers. Then
the rank of $(S^{2n+1},\alpha)$ is equal to 1, and the Reeb flow
induces a circle action.

Let $G(z)=\sum_{j=0}^n z_j^{a_j}$. The set
$$B = \{z\in\C^{n+1}\, :\, F(z,\bar z)=1,\,
G(z)=0\}
$$ is known as Brieskorn manifold and $(B, \alpha)$ is a
co-oriented contact manifold with the Reeb vector field
\eqref{Reeb} (see \cite{Lutz}). Therefore, obviously, the system
is integrable in a noncommutative sense.

Note that
$$
Z(G)=4i\cdot G,
$$
implies
\begin{equation}\label{G0}
[G_1,1]=[G_2,1]=0\vert_B,
\end{equation}
as well as $[G_1,G_2]\ne 0\vert_B$ (see \eqref{nenula}). Here, by
$G_1$ and $G_2$ we denoted the real and the imaginary part of $G$:
\begin{eqnarray*}
&&G_1(z,\bar
z)=\frac12\sum_{j=0}^n\big(z_j^{a_j}+{\bar{z}_j}^{a_j}\big)=\Re(G),\\
&&G_2(z,\bar
z)=\frac1{2i}\sum_{j=0}^n\big(z_j^{a_j}-{\bar{z}_j}^{a_j}\big)=\Im(G).
\end{eqnarray*}

Therefore, the construction presented in Theorem \ref{opsta}
provides an alternative proof of the fact that $(B, \alpha)$ is a
co-oriented contact manifold.

In what follows we shall describe the Jacobi brackets $[f,g]_B$
within the Lie algebra of integrals of the Reeb flow.

\begin{lem}\label{pom}
Define
\begin{eqnarray*}
&&V_1=2i\sum_{j=0}^n\left({\bar{z}_j}^{a_j-1}\dzj-z_j^{a_j-1}\dzkj\right),\\
&&V_2=-2\sum_{j=0}^n\left({\bar{z}_j}^{a_j-1}\dzj+z_j^{a_j-1}\dzkj\right).
\end{eqnarray*}
Then $V_1(z),\, V_2(z)$ are tangent to $S^{2n+1}$ for all $z\in
B$. The restrictions to $B$ of the contact Hamiltonian vector
fields $Y_{G_j}$ on $(S^{2n+1},\alpha)$ are given by $\hat
V_j\vert_B$, $j=1,2$.
\end{lem}

\begin{proof}
We have
$$
dF(V_1)=2i\sum_{j=0}^n\big({\bar{z}_j}^{a_j}-z_j^{a_j}\big)=0\vert_B,\qquad
dF(V_2)=-2\sum_{j=0}^n\big({\bar{z}_j}^{a_j}+z_j^{a_j}\big)=0\vert_B,
$$
which proves the first assertion. Now,
\[
\begin{array}{rcl}
i_{V_1}d\alpha &=& \ds\frac{i}4\sum_{j=0}^na_j dz_j\wedge d\zkj(V_1,\cdot)\\[2ex]
&=& \ds\frac{i}4\sum_{j=0}^na_j\big(dz_j(V_1)d\zkj-d\zkj(V_1)dz_j\big)\\[2ex]
&=& \ds-\frac12\sum_{j=0}^na_j\big({\bar{z}_j}^{a_j-1}d\zkj+{z_j}^{a_j-1}dz_j\big)\\[2ex]
&=& -dG_1,
\end{array}
\]
and similarly for $dG_2$.
\end{proof}

\begin{thm}\label{treca}
Let $f$ and $g$ be integrals of the Reeb vector field
\eqref{Reeb}. Then
\[
[f,g]_B=[f,g]+\ds\frac{df(V_2)dg(V_1)-df(V_1)dg(V_2)}\mu,
\]
where $[\cdot,\,\cdot]_B$ is the Jacobi bracket on $(B,\alpha)$,
$[\cdot,\,\cdot]$ is the Jacobi bracket on $(S^{2n+1},\alpha)$ and
$$
\mu=2\sum_{j=0}^na_j|z_j|^{2(a_j-1)}=2\sum_{j=0}^n a_j
f_j^{a_j-1}\neq0.
$$
\end{thm}

\begin{proof}
Let $Y_f, Y_{G_1}, Y_{G_2}$ be the contact Hamiltonian vector
fields on $(S^{2n+1},\alpha)$ and
$$
W_f=Y_f-\lambda_1 Y_{G_1}-\lambda_2 Y_{G_2}.
$$

Taking into account \eqref{der}, \eqref{G0}, and  Lemma \ref{pom},
we have
\begin{equation}\label{nenula}
[G_1,G_2]=Y_{G_1}G_2=dG_2(Y_{G_1})=dG_2(V_1)=\mu\neq0\vert_B.
\end{equation}

Theorem \ref{opsta} implies that the system
\begin{equation}\label{W}
dG_1(W_f)=dG_2(W_f)=0,
\end{equation}
has a unique solution and $Y^*_f=W_f$ is the contact Hamiltonian
vector field on $(B,\alpha)$. In particular, \eqref{W} leads to
\begin{equation}\label{lambda}
\ds\lambda_1=\frac{dG_2(Y_f)}\mu,\enspace
\lambda_2=-\frac{dG_1(Y_f)}\mu
\end{equation}
and since $Z(f)=Z(g)=0$, we get
\[
\begin{array}{rcl}
[f,g]_B&=& [f,g]+\ds\left(-\frac{dG_2(Y_f)}\mu Y_{G_1}+\frac{dG_1(Y_f)}\mu Y_{G_2}\right)g\\[2ex]
&=& [f,g]+\ds\left(-\frac{dG_2(Y_f)}\mu V_1+\frac{dG_1(Y_f)}\mu
V_2\right)g.
\end{array}
\]

Finally, from
$$
dG_j(Y_f)=-d\alpha(Y_{G_j},
Y_f)=-d\alpha(V_j,Y_f)=d\alpha(Y_f,V_j)=-df(V_j), \quad j=1,2,
$$
we conclude the proof.
\end{proof}

\begin{cor}\label{nekomutiranje}
The integrals \eqref{integr} do not commute on $(B,\alpha)$.
\end{cor}

\begin{proof}
From
$$
df_j(V_1)=2i\big({\zkj}^{a_j}-z_j^{a_j}\big),\qquad
df_j(V_2)=-2\big({\zkj}^{a_j}+z_j^{a_j}\big)
$$
and $[f_j,f_k]=0$, we obtain
\begin{eqnarray*}
[f_j,f_k]_B &=& \frac{4i}\mu\Big[\big({\zkj}^{a_j}-z_j^{a_j}\big)
\big({{\bar z}_k}^{a_k}+z_k^{a_k}\big)-
\big({\zkj}^{a_j}+z_j^{a_j}\big) \big({{\bar
z}_k}^{a_k}-z_k^{a_k}\big)\Big]\\
&=& \frac{8i}\mu\Big[ \bar z_j^{a_j}z_k^{a_k}-z_j^{a_j}\bar
z_k^{a_k} \Big]\neq0,
\end{eqnarray*}
for $j\neq k$.
\end{proof}

\begin{prop}\label{potpunost}
The complete noncommutative set of integrals of the Reeb flow on
the Brieskorn manifold $(B, \alpha)$ is given by $f_j$,
$[f_j,f_k]_B$, $j,k=0,\dots,n$.
\end{prop}

\begin{proof}
The integrals $f_j$ and
$$
f_{j,k}=\frac{\mu}8[f_j,f_k]_B=i\Big[ \bar
z_j^{a_j}z_k^{a_k}-z_j^{a_j}\bar z_k^{a_k} \Big]
$$
provide a complete set of noncommuting integrals for the Reeb flow
on $S^{2n+1}$ (their level sets are the Reeb circles). Therefore,
after restricting them to the Brieskorn manifold $B$, we get a
complete set of integrals for a Reeb flow on $B$.

Indeed, consider a torus \eqref{torusi} and a coordinate change
\eqref{polar}. The integrals $f_{j,k}$, restricted to $T_c$,
$$
f_{j,k}\vert_{T_c}=i c_j^{a_j}
c_k^{a_k}\Big[e^{i(a_k\varphi_k-a_j\varphi_j)}-e^{-i(a_k\varphi_k-a_j\varphi_j)}
\Big],
$$
correspond to the rational relations
$$
\rho_{j,k}:  \quad m_j \omega_j+m_k \omega_k=0,\quad m_j=a_j,
\quad m_k=a_k.
$$

Since among $\rho_{j,k}$ we have $n$ independent relations, among
$f_{j,k}\vert_{T_c}$ we have $n$ independent functions. Thus,
among $f_j, f_{j,k}$ we have $2n$ independent functions on
$S^{2n+1}$, that is $2n-2$ independent functions on $B$.
\end{proof}

In \cite{U}, Ustilovsky studied the Brieskorn manifolds $(B_p,
\alpha_p)$ with
\begin{equation}\label{ustilovsky}
a_0=p, \quad a_1=\dots=a_n=2,
\end{equation}
where $n=2m+1$ and $p\equiv\pm 1\ (\mod 8)$. The manifold $B_p$ is
diffeomorphic to a standard sphere $S^{4m+1}$ \cite{Br}, and as
shown in \cite{U}, for $p_1\ne p_2$, the contact structures
$\mathcal H_{p_1}=\ker\alpha_{p_1}$ and $\mathcal
H_{p_1}=\ker\alpha_{p_1}$ are not isomorphic. The proof is based
on the study of periodic trajectories of the Reeb flow of the
perturbed contact form $\frac{1}{H}\alpha_p$, which is equal to
the contact flow
\begin{equation}\label{BU}
\dot z=Y^*_H
\end{equation}
on $(B_p, \alpha_p)$, where
\begin{eqnarray*}
&& H=F+\sum_{j=1}^m \epsilon_j g_j, \qquad 0<\epsilon_j<1, \quad
j=1,\dots,m,\\
&& g_j=i(\bar z_{2j}z_{2j+1}-z_{2j}\bar
z_{2j+1})=2(y_{2j}x_{2j+1}-y_{2j+1}x_{2j}).
\end{eqnarray*}

From the point of view of integrability, we can consider the contact
flow of $H$ as an integrable perturbation of the Reeb flow.

\begin{prop} \label{UB}
The contact flow \eqref{BU} is completely integrable in a
noncommutative sense. Generic invariant pre-isotropic tori are of
dimension $m+1$, spanned by the Reeb flow and the contact flows of
integrals $g_1,\dots,g_m$.
\end{prop}

\begin{proof}
 Since $f_i, g_j$ are integrals of the system
of harmonic oscillators with conditions \eqref{ustilovsky}, after
the coordinate transformation \eqref{promena}, from Proposition
\ref{prva} we obtain the following commuting relations on
$(S^{4m+3},\alpha_p)$:
\begin{eqnarray*}
&& [g_i,g_j]=\{g_j,g_i\}\vert_{E_\frac14}=0, \\
&& [g_j,1]=\{H_0,g_j\}\vert_{E_\frac14}=0, \\
&& [f_0,g_j]=\{g_j,f_0\}\vert_{E_\frac14}=0, \\
&& [f_1,g_j]=\{g_j,f_1\}\vert_{E_\frac14}=0, \\
&& [h_i,g_j]=\{g_j,h_i\}\vert_{E_\frac14}=0,
\end{eqnarray*}
where $h_j=f_{2j}+f_{2j+1}$, $j=1,\dots,m$.

Next, since
\begin{eqnarray*}
dg_j(V_1) &=& i(z_{2j+1}d\bar z_{2j}+\bar z_{2j}dz_{2j+1}-
z_{2j+1}d\bar z_{2j}-\bar z_{2j}dz_{2j+1})(V_1)\\
 &=& -2(-z_{2j+1}z_{2j}+\bar z_{2j}\bar z_{2j+1}+z_{2j+1}z_{2j}-\bar z_{2j}\bar
 z_{2j+1})=0,
\end{eqnarray*}
\begin{eqnarray*}
dg_j(V_2) &=& i(z_{2j+1}d\bar z_{2j}+\bar z_{2j}dz_{2j+1}-
z_{2j+1}d\bar z_{2j}-\bar z_{2j}dz_{2j+1})(V_2)\\
 &=& -2i(z_{2j+1}z_{2j}+\bar z_{2j}\bar z_{2j+1}-z_{2j+1}z_{2j}-\bar z_{2j}\bar
 z_{2j+1})=0,
\end{eqnarray*}
from Theorem \ref{treca} we get the commuting relations on
$(B_p,\alpha_p)$ as well:
\begin{eqnarray*}\label{CU}
&& [g_i,g_j]_{B_p}=0, \quad [f_0,g_j]_{B_p}=0, \quad
[f_1,g_j]_{B_p}=0, \\
&& [h_i,g_j]_{B_p}=0,  \quad  [g_j,1]_{B_p}=0.
\end{eqnarray*}

Within the set of integrals of the Reeb flow, the Jacobi bracket
of two integrals of the system \eqref{BU} is both the integral of
\eqref{BU} and of the Reeb flow. Thus, we have a set of integrals
\begin{equation}\label{integrali-U}
g_j, \quad h_j=f_{2j}+f_{2j+1}, \quad f_0, \quad f_1, \quad
[f_0,h_j]_{B_p}, \quad [f_1,h_j]_{B_p}, \quad [h_i,h_j]_{B_p},
\end{equation}
that commute with $g_k$, for all $i,j,k=1,\dots,m$.

For the noncommutative integrability, it remains to note that
among the integrals \eqref{integrali-U} we have $3m$ independent
functions, including $g_1,\dots,g_m$. Indeed,
\begin{equation}\label{maple}
dG_1\wedge dG_2\wedge dF\wedge dg_1\wedge \dots \wedge dg_m\wedge
dh_1 \wedge \dots \wedge dh_m \wedge dq_1\wedge \dots \wedge
dq_m\ne 0
\end{equation}
holds on an open dense subset of $B_p$. Here, since $\mu$ is an
integral of the system, instead of $[f_1,h_j]_{B_p}$ we consider as in Proposition \ref{potpunost} the integrals
$$
q_j=\frac{\mu}8[f_1,h_j]_{B_p}=i\Big[ \bar
z_1^{2}z_{2j}^{2}-z_1^{2}\bar z_{2j}^{2}\Big]+i\Big[ \bar
z_1^{2}z_{2j+1}^{2}-z_1^{2}\bar z_{2j+1}^{2}\Big], \quad
j=1,\dots,m.
$$

The relation \eqref{maple} can be verified by straightforward
calculations in the polar coordinates \eqref{polar}.
\end{proof}

\section{Note on $K$-contact manifolds}

The integrability of the Reeb flow on a sphere (Proposition
\ref{sfera}) is a particular case of the integrability of the Reeb
flows on compact $K$-contact manifolds.

Let $(M,\alpha)$ be a $(2n+1)$-dimensional co-oriented contact
manifold. Then $d\alpha$ induces a symplectic structure on the
contact distribution $\mathcal H=\ker\alpha$. In this situation,
there exist a positive definite metric $g_\mathcal H$ and an
almost complex structure $J$ on $\mathcal H$, such that
$$
g_\mathcal H(X,Y)=d\alpha(X,JY), \quad g_\mathcal
H(JX,JY)=g_\mathcal H(X,Y),
$$
for all horizontal vector fields $X$ and $Y$.

The metric $g:= g_\mathcal H\oplus(\alpha\otimes\alpha)$ on $M$ is
called {\it adapted metric} to the contact form $\alpha$. If
there exists an adapted metric $g$, such that the Reeb vector field
is a Killing vector field
$$
\mathcal L_Z g=0,
$$
then we call $(M,\alpha,g)$ a {\it $K$-contact manifold}. For
example, the Sasakian manifolds have a natural $K$-contact
structure (e.g., see \cite{BG}).

The following statement is a simple modification of Proposition
2.1 in Yamazaki \cite{Y}.

\begin{prop}\label{K-contact}
The Reeb flow on a compact $K$-contact manifold $(M,\alpha,g)$ is
(noncommutative) contact integrable.
\end{prop}

\begin{proof}
The Reeb flow $\varphi(t)$ is a one-parametric subgroup in a
compact Lie group $G$ of isometries of $(M,g)$. It follows that
the closure of $\{\varphi(t)\,\vert\, t\in\R$\} in $G$ is a torus
$\T^{r+1}$ for some integer $r$. Thus, from
$\varphi(t)^*\alpha=\alpha$, we have $s^*\alpha=\alpha$ for all
$s\in\T^{r+1}$ and $(M,\alpha)$ is a $\T^{r+1}$-contact manifold
(the number $r+1$ is called the {\it rank} of $(M,\alpha,g)$
\cite{Y, BG}).

Let $\xi$ be in the Lie algebra $\mathfrak t^{r+1}$ of $\T^{r+1}$.
It induces a contact vector field $Y_\xi$ on $M$ with a
Hamiltonian $f_\xi=\alpha(Y_\xi)$ ($\mathbf f: M\to (\mathfrak
t^{r+1})^*$, $\mathbf f(\xi):=f_\xi$ is known as a contact
momentum mapping \cite{Le, BG}). Take a base $\xi_0,\dots,\xi_{r}$
of $\mathfrak t^{r+1}$ such that $Z=Y_{\xi_{0}}$, that is
$f_{\xi_0}\equiv 1$. Since $\T^{r+1}$ is Abelian, we have
$[Y_{\xi_i},Y_{\xi_j}]=0$ and $ [f_{\xi_i},f_{\xi_j}]=0$,
$i,j=0,\dots n$.

In a neighborhood of a generic point, the $\T^{r+1}$-action
defines a foliation $\mathcal F$ on invariant $r+1$-dimensional
tori. In particular, according to \eqref{der}, we have
$d\alpha(\hat{Y}_{\xi_i},\hat{Y}_{\xi_j})=0$, so the leaves of
$\mathcal F$ are pre-isotropic and $r\le n$. Let
$f_{r+1},\dots,f_{2n-r}$ be the integrals of $\mathcal F$ defined
on an open dense set of $M$. Since $\mathcal
L_{Y_{\xi_j}}(f_i)=0$, we have
$$
[f_i,1]=0, \quad [f_i,f_{\xi_j}]=0, \quad i=r+1,\dots,2n-r, \quad
j=1,\dots,r,
$$ which completes the proof.
\end{proof}

\subsection*{Acknowledgments}
We are grateful to the referee for useful remarks. The research of
B. J. was supported by the Serbian Ministry of Science Project
174020, Geometry and Topology of Manifolds, Classical Mechanics
and Integrable Dynamical Systems.

\end{document}